\documentclass{amsart}

\usepackage{amssymb,amsfonts,amsxtra}
\usepackage{paralist,longtable}
\usepackage{dsfont,mathrsfs}
\renewcommand{\mathcal}{\mathscr}
\usepackage[all]{xypic}
\xyoption{dvips}

\theoremstyle{definition}
\newtheorem{ntn}{Notation}[section]
\newtheorem{dfn}[ntn]{Definition}
\theoremstyle{plain}
\newtheorem{lem}[ntn]{Lemma}
\newtheorem{prp}[ntn]{Proposition}
\newtheorem{thm}[ntn]{Theorem}

\newtheorem{cnj}[ntn]{Conjecture}

\theoremstyle{remark}

\newtheorem{rmk}[ntn]{Remark}
\newtheorem{exa}[ntn]{Example}

\renewcommand{\aa}{\mathfrak{a}}
\newcommand{\bb}{\mathfrak{b}}
\newcommand{\calD}{\mathcal{D}}
\newcommand{\calF}{\mathcal{F}}
\newcommand{\calO}{\mathcal{O}}
\newcommand{\CC}{\mathbb{C}}
\newcommand{\cc}{\mathfrak{c}}
\renewcommand{\d}{\mathbf{d}}
\newcommand{\eps}{\epsilon}
\renewcommand{\gg}{\mathfrak{g}}
\newcommand{\hh}{\mathfrak{h}}
\newcommand{\gl}{\mathfrak{gl}}
\newcommand{\ideal}[1]{{\langle#1\rangle}}
\newcommand{\into}{\hookrightarrow}
\newcommand{\NN}{\mathbb{N}}
\newcommand{\ol}{\overline}

\newcommand{\p}{\partial}
\newcommand{\PP}{\mathbb{P}}
\newcommand{\QQ}{\mathbb{Q}}
\newcommand{\R}{\mathbf{R}}
\newcommand{\red}{\mathrm{red}}
\newcommand{\RR}{\mathbb{R}}

\renewcommand{\ss}{\mathfrak{s}}
\newcommand{\veps}{\varepsilon}
\newcommand{\wt}{\widetilde}
\newcommand{\xymat}{\SelectTips{cm}{}\xymatrix}
\newcommand{\ZZ}{\mathbb{Z}}

\DeclareMathOperator{\Ann}{Ann}
\DeclareMathOperator{\Der}{Der}

\DeclareMathOperator{\GL}{GL}
\DeclareMathOperator{\gr}{gr}
\DeclareMathOperator{\grad}{grad}
\DeclareMathOperator{\Hom}{Hom}
\DeclareMathOperator{\id}{id}
\DeclareMathOperator{\ord}{ord}
\DeclareMathOperator{\Rep}{Rep}
\DeclareMathOperator{\rk}{rk}
\DeclareMathOperator{\SL}{SL}
\DeclareMathOperator{\Spec}{Spec}
\DeclareMathOperator{\Sym}{Sym}
\DeclareMathOperator{\tr}{tr}
\DeclareMathOperator{\Var}{Var}

\numberwithin{equation}{section}

\usepackage{amsmath}

\begin{document}

\title[$b$-functions of linear free divisors]{On the symmetry of $b$-functions\\of linear free divisors}

\author{Michel Granger}
\address{
M.~Granger\\
D\'epartement de Math\'ematiques\\
Universit\'e d'Angers\\
2 Bd Lavoisier\\
49045 Angers\\
France
}
\email{granger@univ-angers.fr}
\thanks{}

\author{Mathias Schulze}
\address{
M.~Schulze\\
Oklahoma State University\\
Department of Mathematics\\
Stillwater, OK 74078\\
United States}
\email{mschulze@math.okstate.edu}
\thanks{MS was supported by the College of Arts \& Sciences at Oklahoma State University through a FY 2009 Dean's Incentive Grant.}

\date{\today}



\begin{abstract}
We introduce the concept of a prehomogeneous determinant as a nonreduced version of a linear free divisor.
Both are special cases of prehomogeneous vector spaces.
We show that the roots of the $b$-function are symmetric about $-1$ for reductive prehomogeneous determinants and for regular special linear free divisors.
For general prehomogeneous determinants, we describe conditions under which this symmetry still holds.

Combined with Kashiwara's theorem on the roots of $b$-functions, our symmetry result shows that $-1$ is the only integer root of the $b$-function. 
This gives a positive answer to a problem posed by Castro-Jim\'enez and Ucha-Enr\'\i quez in the above cases.

We study the condition of (strong) Euler homogeneity in terms of the action of the stabilizers on the normal spaces.

As an application of our results, we show that the logarithmic comparison theorem holds for Koszul free reductive linear free divisors exactly if they are (strongly) Euler homogeneous.
\end{abstract}

\maketitle
\tableofcontents

\section{Linear free divisors and prehomogeneous determinants}

Linear free divisors have been introduced in \cite{BM06} in the context of quiver representations and have been studied in \cite{GMNS06} and \cite{GMS08}.

A reduced hypersurface $D$ in an $n$-dimensional complex vector space $V$ is called a \emph{linear free divisor} if the module $\Der(-\log D)$ of \emph{logarithmic vector fields} has a basis of global degree $0$ vector fields
\[
\delta_k=A_kx=x^tA_k^t\p_x\in\Gamma(V,\Der(-\log D))_0,\quad A_k\in\CC^{n\times n},\quad k=1,\dots,n.
\]
Then, by Saito's criterion \cite[Thm.~1.8.ii]{Sai80},
\[
f=\det(\delta_1,\dots,\delta_n)\in\Sym(V^*),
\]
is a homogeneous defining equation for $D$ of degree $n$.
Therefore the can assume that $\delta_1=\eps$ is the \emph{Euler vector field}
\[
\eps=\frac{1}{n}\sum_{i=1}^nx_i\p_{x_i}
\]
and then also that $\delta_2,\dots,\delta_n\in\Gamma(V,\Der(-\log f))_0$ where $\Der(-\log f)$ is the module of vector fields tangent to (the level sets of) $f$.

We denote by $G_D\subseteq\GL_V$ the unit component of the stabilizer of $f\in\PP\Sym(V^*)$, and by $A_D\subseteq G_D$ that of $f\in\Sym(V^*)$.
The following was shown in \cite[Lems.~2.2--2.4]{GMNS06}: 
$G_D$ is a linear algebraic group with Lie algebra of infinitesimal generators $\gg_D=\Gamma(V,\Der(-\log D))_0$, and the same hold for $A_D$ with $\aa_D=\Gamma(V,\Der(-\log f))_0$.
By Saito's criterion, $V\backslash D$ is a Zariski open orbit with finite isotropy groups and $G_D$ an $n$-dimensional group.
Conversely, any reduced discriminant of an $n$-dimensional linear algebraic group in $\GL_V$ with an open orbit is a linear free divisor.
Thus, linear free divisors are special cases of discriminants of prehomogeneous vector spaces $(G,\rho,V)$ characterized by the following conditions:
\begin{enumerate}
\item $G\subseteq\GL_V$ is connected with $\dim G=\dim V$,
\item $\rho=\id:G\into\GL_V$,
\item $f(x)=\det(\rho(A_1)x,\dots,\rho(A_n)x)\in\Sym(V^*)$, $\gg=\ideal{A_1,\dots,A_n}$, is reduced.
\end{enumerate}
We shall consider the more general objects defined by dropping the third condition.
However, we shall frequently impose the additional condition of reductivity on $G$ when we study $b$-functions in Section~\ref{65}.

\begin{dfn}\label{58}
Let $G\subseteq\GL_V$ be a connected algebraic group with $\dim G=\dim V=n$ such that $(G,\rho,V)$ is a prehomogeneous vector space where $\rho=\id:G\into\GL_V$. 
Consider a relative invariant 
\begin{equation}\label{3}
f(x)=\det(d\rho(A_1)x,\dots,d\rho(A_n)x)\in\Sym(V^*)
\end{equation}
defined by generators $A_1,\dots,A_n$ of the Lie algebra $\gg$ of $G$.
We call the (possibly nonreduced) discriminant divisor $D\subseteq V$ defined by $f$ a \emph{prehomogeneous (discriminant) determinant}, and a \emph{linear free divisor} in case $f$ is reduced.
We call $D$ \emph{reductive} if $G$ is a reductive group.
We set $G_D=G$, $\gg_D=\gg$, and $\rho_D=\rho$, and we denote the character of $f$ by $\chi_D\colon G_D\to\CC^*$ and its derivative by $d\chi_D\colon\gg_D\to\CC$.
We call 
\begin{equation}\label{2}
A_D=(\ker\chi_D)^\circ,\quad\aa_D=\ker d\chi_D,
\end{equation}
the \emph{annihilator} of $D$ where $^\circ$ denotes the unit component.
\end{dfn}

While $G_D$ is determined by $D$ if $D$ is a linear free divisor, it is part of the data of a general prehomogeneous determinant $D$.
Note that, in Definition~\ref{58}, $f$ is indeed a relative invariant by \cite[Thm.~2.9]{Kim03}.
Then $\chi_D$ is not trivial by \cite[Prop.~2.4.(2)]{Kim03} and we have
\begin{equation}\label{83}
G_D/A_D\cong\CC^*,\quad\gg_D/\aa_D\cong\CC.
\end{equation}
In particular, we may pick $A_1\in\gg_D$ such that
\begin{equation}\label{55}
d\chi_D(A_1)=1.
\end{equation}
We choose $A_1=E/n$ if $D$ is a linear free divisor.
For $A\in\gg_D$, we abbreviate 
\begin{equation}\label{21}
\delta_A=\ideal{d\rho_D(A)x,\p_x}=x^tA^t\p_x,\quad\delta_k=\delta_{A_k}.
\end{equation}
By definition of a relative invariant \cite[\S2.2]{Kim03} and \cite[Lem.~2.15]{Kim03}, 
\begin{equation}\label{20}
f\circ\rho_D(g)=\chi_D(g)\cdot f,\quad\delta_A(f)=\ideal{d\rho_D(A)x,\grad_x(f)}=d\chi_D(A)\cdot f.
\end{equation}
As $\deg f=n$ by \eqref{3} and $G_D\subseteq\GL_V$ with $\dim V=n$ by Definition~\ref{58}, we have
\begin{equation}\label{1}
\deg\chi_D=n=\deg(\det\circ d\rho_D).
\end{equation}

\begin{lem}\label{33}
For a prehomogeneous determinant $D$, $\chi_D=\det\circ\rho_D$ if and only if $\tr\circ d\rho_D=0$ on $\aa_D$.
\end{lem} 

\begin{proof}
The converse being trivial, let us assume that the second condition holds.
As $d\det=\tr$, this gives using \eqref{2} that
\[
\ker d(\det\circ\rho_D)\supseteq\aa_D=\ker d\chi_D, \quad\ker(\det\circ\rho_D)\supseteq A_D=\ker\chi_D.
\]
Then both $\chi_D$ and $\det\circ\rho_D$ induce characters on $G_D/A_D$.
But the latter has character group isomorphic to $\ZZ$ by \eqref{83}.
Using \eqref{1}, this implies the equality $\chi_D=\det\circ\rho_D$.
\end{proof}

\begin{lem}\label{80}
If $D$ is a linear free divisor or a reductive prehomogeneous determinant, one can choose $A_1\in\gg_D$ with \eqref{55} such that
\begin{equation}\label{53}
\gg_D=\CC\cdot A_1\oplus\aa_D,\quad\aa_D=\ideal{A_2,\dots,A_n}.
\end{equation}
\end{lem}

\begin{proof}
Choosing $A_1=E/n$, the first statement follows from \eqref{3} and Saito's criterion.

By \cite[Ch.~III, \S 7, Thm.~10]{Jac79}, reductivity of $\gg_D$ means that
\begin{equation}\label{81}
\gg_D=\cc_D\oplus\ss_D
\end{equation}
where $\cc_D$ is the center of $\gg_D$, $d\rho_D(\cc_D)$ consists of semisimple endomorphisms, and $\ss_D$ is a semisimple ideal.
By semisimplicity of $\ss_D$, we have
\begin{equation}\label{57}
\ss_D=[\ss_D,\ss_D].
\end{equation}
On the other hand, it follows from \eqref{20}, \eqref{2}, and \eqref{55} that
\begin{equation}\label{82}
[\gg_D,\gg_D]\subseteq\aa_D
\end{equation}
and hence $A_1\not\in\ss_D$ by \eqref{57}.
Using \eqref{81} and \eqref{82}, this shows that $A_1\in\cc_D$ after subtracting a suitable element of $\ss_D$.
\end{proof}

In order to say more, we shall focus on the reductive case for the remainder of this section.
The following result is proved in \cite[Cor.~4.4]{GMS08} in the case of linear free divisors.

\begin{thm}\label{56}
Let $D$ be a reductive prehomogeneous determinant. 
Then 
\begin{equation}\label{27}
\chi_D=\det\circ\rho_D,\quad d\chi_D=\tr\circ d\rho_D
\end{equation}
which means that $A_D\subseteq\SL_V$.
\end{thm}

\begin{proof}
We resume the notation $\cc_D$, $\ss_D$, and $\hh_D$ from Lemma~\ref{80}.
By \eqref{57} and \eqref{82}, $\ss_D\subseteq\aa_D$ and hence $\aa_D=\cc'_D\oplus\ss_D$ where $\cc'_D$ is the center of $\aa_D$.

Now let $\hh_D$ be a Cartan subalgebra of $\ss_D$.
Recall that $A_1\in\cc_D$ by Lemma~\ref{80}.
Pick generators $S_1=A_1,S_2,\dots,S_s$ of $\cc_D+\hh_D$ such that $d\rho_D(S_1),\dots,d\rho_D(S_s)$ are semisimple and extend them to a basis $N_1,\dots,N_r$, $r=n-s$, of $\aa_D$ such that $N_1,\dots,N_r$ are $(\cc_D+\hh_D)$-homogeneous and $d\rho_D(N_1),\dots,d\rho_D(N_r)$ are nilpotent.
Then, by (the proof of) \cite[Thm.~6.1.(4)]{GMNS06}, we have
\begin{equation}\label{79}
d\chi_D(S_i)=\tr\circ d\rho_D(S_i)+\sum_{j=1}^r[S_i,N_j]/N_j,\quad i=1,\dots,s.
\end{equation}
This implies that $d\chi_D=\tr\circ d\rho_D$ on $\cc_D$.
But on $\ss_D$, both sides of this equation are zero by \eqref{57}.
Now the claim follows from Lemma~\ref{33}.
\end{proof}

Our motivation to study the symmetry of $b$-functions of prehomogeneous determinants stems from M.~Sato's fundamental theorem of prehomogeneous vector spaces \cite[Prop.~4.18]{Kim03}.
We shall need the following result to show that it holds for reductive prehomogeneous determinants in Theorem~\ref{64}.

\begin{thm}\label{4}
Let $D$ be a reductive linear free divisor. 
\begin{enumerate}[(a)]
\item\label{4a} The Lie algebra representation $\rho_D$ is defined over $\QQ$ with respect to some $V_\QQ$.
\item\label{4b} With respect to any $V_\QQ$ as in \eqref{4a}, also $D$ is defined over $\QQ$.
\item\label{4c} With respect to any $V_\QQ$ as in \eqref{4a}, also $G_D$ is defined over $\QQ$.
\end{enumerate}
\end{thm}

For the convenience of the reader we give a proof of the following general fact which is probably well-known to specialists.

\begin{lem}\label{15}
Every finite representation $V$ of a complex semisimple Lie algebra $\ss$ is defined over $\QQ$ with respect to some $V_\QQ$.
\end{lem}

\begin{proof}
Let $\hh$ be a Cartan and $\bb$ be a Borel subalgebra of $\ss$.
By Chevalley's Normalization \cite[Ch.~VI, \S6, Thm.~11]{Ser01} $\ss$ has a rational form $\ss_\QQ$.
By complete reducibility we may assume that $V$ is an irreducible $\ss$-module.
Let $V_\omega=U(\ss)\otimes_{\bb}L_\omega$ where $U(\ss)$ denotes the universal enveloping algebra of $\ss$ and $L_\omega$ is a complex one dimensional $\bb$-module of weight $\omega\in\hh^*$.
Then set $E_\omega=V_\omega/N_\omega$ where $N_\omega$ is generated by all $\ss$-modules strictly contained in $V_\omega$.
By \cite[Ch.~VII, \S3, Thm.~2]{Ser01}, $V=E_\omega$ where $\omega$ is integer on the coroots in $\hh_\QQ$ by \cite[Ch.~VII, \S4, Prop.~3.(b)]{Ser01}.
This implies that $V_\omega$ is a direct sum of rational weight spaces for $\hh_\QQ$ and is thus defined over $\QQ$.
Then $N_\omega$ is the maximal $\ss$-module in $V_\omega$ that meets $L_\omega$ in zero.
Choosing a $\QQ$-basis of $L_\omega$, $V_\omega$ also decomposes as a direct sum of $\QQ$-vector spaces $U(\ss_\QQ)\otimes_{\bb_\QQ}L_{\omega,\QQ}$ compatible with the above direct sum decomposition.
Clearly $N_\omega$ is homogeneous with respect to both direct sums and hence defined over $\QQ$.
\end{proof}

\begin{proof}[Proof of Theorem~\ref{4}]
We show that $\rho_D$ is defined over $\QQ$ when restricted to $\aa_D$, which implies \eqref{4a} by \eqref{53} and $A_1=E/n$.
Then \eqref{4b} follows by \eqref{3}, and \eqref{4c} as in the proof of \cite[Lem.~2.2]{GMNS06}.

We resume the notation $\cc'_D$, $\ss_D$, and $\hh_D$ from the proof of Theorem~\ref{56}.
As $\cc'_D\oplus\hh_D$ is commutative, its representation on $V$ can be diagonalized.
By Lemma~\ref{15}, the representation of $\ss_D$ on each $\cc'_D$-weight space of $V$ is defined over $\QQ$.
So there is a $\QQ$-subspace $V_\QQ$ of $V$ with respect to which $\rho_D|_{\ss_D}$ is defined over $\QQ$ and with respect which $\rho_D(\cc'_D\oplus\hh_D)$ consists of semisimple endomorphisms (with possibly complex weights).
As $\CC\cdot A_1\oplus\cc'_D\oplus\hh_D$ is a Cartan subalgebra of $\gg_D$, it follows from \cite[Ch.~III, \S5, Thm.~3.(b)]{Ser01} that it equals its centralizer.
Then \cite[Lem.~1.4]{Sai71} shows that $\rho_D$ is defined over $\QQ$ with respect to $V_\QQ$ on $\CC\cdot A_1\oplus\cc'_D\oplus\hh_D$.
The map $A\mapsto A-d\chi_D(A)\cdot A_1$ from $\CC\cdot A_1\oplus\cc'_D\oplus\hh_D$ to $(\CC\cdot A_1\oplus\cc'_D\oplus\hh_D)\cap\aa_D=\cc'_D\oplus\hh_D$ preserves the property of having eigenvalues in $\QQ$.
Thus, $\rho_D$ is defined over $\QQ$ with respect to $V_\QQ$ also on $\cc'_D\oplus\hh_D$ and hence on $\aa_D$.
Remark that once $\rho_D$ is defined over $\QQ$ on $\gg_D$, this also holds on the center $\CC\cdot A_1\oplus\cc'_D$ of $\gg_D$.
\end{proof}

\section{$b$-functions of reductive prehomogeneous determinants}\label{65}

The (global) $b$-function of $f\in\CC[x]$ is the minimal monic polynomial $B_f(s)\in\CC[s]$ in a functional equation
\begin{equation}\label{7}
B_f(s)\cdot f^s\in D_V[s]\cdot f^{s+1}
\end{equation}
where $D_V$ denotes the ring of algebraic differential operators on $V$.
In general, the existence of the $b$-function has been established by I.N.~Bernstein~\cite{Ber72} in the global and by J.E.~Bj\"ork \cite{Bjo79} in the local case.
M.~Kashiwara \cite{Kas76} proved that the roots of the $b$-function are negative rational numbers.

Let $(G,\rho,V)$ be a reductive prehomogeneous vector space.
Then there is a compact Zariski dense subgroup $H\subseteq G$ and one can assume that $H\subseteq U(n)$ for some basis $x$ of $V$.
We refer to this basis $x$ and the dual basis $y$ of $V^*$ as unitary coordinates.
For a relative invariant $f$ of $(G,\rho,V)$ with character $\chi_f$, the function $f^*\in\Sym(V)$ defined by
\begin{equation}\label{59}
f^*(y)=\ol{f(\ol y)}
\end{equation}
is then a relative invariant of $(G,\rho^*,V^*)$ with character
\begin{equation}\label{61}
\chi_{f^*}=\chi_f^{-1}
\end{equation}
by the unitarian trick (see \cite[\S 2.3]{Kim03}).
In \cite[Prop.~2.22 and 2.23]{Kim03}, this is used to show that $f^*$ defines functional equations 
\begin{gather}
f^*(\p_x)\cdot f^{s+1}(x)=b_f(s)\cdot f^s(x)\label{8},\\
\nonumber f(\p_y)\cdot(f^*)^{s+1}(y)=b_{f^*}(s)\cdot(f^*)^s(y),\\
b_f(s)=b_{f^*}(s)\label{9},
\end{gather}
as in \eqref{7} and that
\begin{equation}\label{74}
\deg b_f(s)=\deg f=n.
\end{equation}
From this one deduces that $(G,\rho,V)$ is a regular prehomogeneous vector space and that any relative invariant $f$ defining the discriminant is nondegenerate \cite[Prop.~2.24]{Kim03}.
In particular, $(G,\rho^*,V^*)$ is again a reductive and hence a regular prehomogeneous vector space.
It is shown in \cite[Cor.~2.5.10]{Gyo91} that 
\begin{equation}\label{34}
b_f(s)=B_f(s),\quad b_{f^*}(s)=B_{f^*}(s).
\end{equation}

Now let $D$ be a (not necessarily reductive) prehomogeneous determinant defined by $f$ as in Definition~\ref{58}.
Then we abbreviate
\[
B_D(s)=B_f(s),\quad b_D(s)=b_f(s),
\]
where the latter is defined in the presence of the functional equations \eqref{8}.

\begin{dfn}\label{60}
Let $D$ be a prehomogeneous determinant. 
Consider $G_{D^*}=\rho_D^*(G_D)\subseteq\GL_{V^*}$ and $\rho_{D^*}=\id\colon G_{D^*}\into\GL_{V^*}$. 
Define $D^*$ and $f^*$ as $D$ and $f$ in Definition~\ref{58} with $(G,\rho,V)=(G_{D^*},\rho_{D^*},V^*)$.
We call $D^*$ the \emph{dual determinant} of $D$.
Note that $D^*$ is a prehomogeneous determinant exactly if $f^*\ne 0$ or $D^*\ne V^*$.
\end{dfn}

For $A\in\gg_D$, we abbreviate 
\begin{equation}\label{19}
\delta_A^*=\ideal{d\rho_D^*(A)y,\p_y}=-y^tA\p_y,\quad\delta_k=\delta_{A_k},
\end{equation}
as in \eqref{21} using \eqref{55} and \eqref{53}.
Note that \eqref{20} also holds for $f$, $D$, $\delta_A$ replaced by $f^*$, $D^*$, $\delta_A^*$.

The following statement is \cite[Prop.~3.7]{GMS08} in case of linear free divisors.

\begin{prp}\label{62}
With $D$ also $D^*$ is a reductive prehomogeneous determinant (or a reductive linear free divisor) defined by $f^*$ in \eqref{59}.
\end{prp}

\begin{proof}
In unitary coordinates, we have $\rho_D^*(g)=(\rho_D(g)^{-1})^t=\ol{\rho_D(g)}$ and hence $d\rho_D^*(A)=-d\rho_D(A)^t=\ol{d\rho_D(A)}$.
Thus, $f^*$ in Definition~\ref{60} coincides with $f^*$ in \eqref{59}.
This shows that $f^*\ne0$ and that $f^*$ is reduced if and only if $f$ is reduced.
\end{proof}

If $D$ is a reductive prehomogeneous determinant, the two (equivalent) equalities
\begin{gather}\label{17}
\chi_{D^*}\circ\rho_D^*=\chi_{f^*}=\chi_f^{-1}=\chi_D^{-1},\\
\nonumber d\chi_{D^*}\circ d\rho_D^*=d\chi_{f^*}=-d\chi_f=-d\chi_D
\end{gather}
follow from Proposition~\ref{62} and \eqref{61}.

The following observation was the starting point for our study of the symmetry of $b$-functions of prehomogeneous determinants in the remainder of this section.

\begin{thm}\label{64}
M.~Sato's fundamental theorem of homogeneous vector spaces as formulated in \cite[Prop.~4.18]{Kim03} holds true for reductive linear free divisors.
\end{thm}

\begin{proof}
The hypotheses for the proof of this theorem in \cite[\S4.1]{Kim03} are fulfilled by Theorem~\ref{4}.\eqref{4c} except for the requirement that $f$ be irreducible.
But this latter assumption is used in the proof only to ensure that $f$ and $f^*$ can be chosen defined over $\RR$ (see \cite[Prop.~4.1.(1), Prop.~4.2]{Kim03}) and that $|\det\rho(g)|=|\chi_f(g)|^{n/d}$ for all $g\in G$ where $d=\deg f$ (see \cite[p.~119]{Kim03}).
But, by Theorem~\ref{56} and \ref{4}.\eqref{4b}, these statements hold true for any reductive linear free divisor.
Indeed, it follows from Theorem~\ref{4}.\eqref{4a} that the Lie algebra representation $d\rho_{D^*}\colon\gg_{D^*}=d\rho_D^*(\gg_D)\into\gl_{V^*}$ is defined over $\QQ$ with respect to $V_\QQ^*$.
Then the defining equation $f^*$ of the dual linear free divisor $D^*$ is defined over $\QQ$ with respect to $V_\QQ^*$ by Theorem~\ref{4}.\eqref{4b}.
Note also that the nondegeneracy of $f$ and $f^*$ used in \cite[Cor.~4.4]{Kim03} follows from \cite[Prop.~2.24]{Kim03} even in the reducible case.
\end{proof}

A symmetry property of the $b$-function as in Lemma~\ref{24} below is a corollary of Sato's fundamental theorem (see \cite[Prop.~4.19]{Kim03}) for irreducible reductive prehomogeneous vector spaces.
For reductive prehomogeneous determinants, we shall give an independent proof which is more elementary in this setting.
Our argument also applies to regular linear free divisors (see Section~\ref{26}). 

\begin{lem}\label{24}
Let $D$ be a prehomogeneous determinant with $D^*\ne V^*$ such that \eqref{8}, \eqref{9}, \eqref{27}, and \eqref{17} hold true.
Then its $b$-function satisfies
\[
b_D(s)=(-1)^n\cdot b_D(-s-2).
\]
\end{lem}

\begin{proof}
From \eqref{20}, \eqref{27}, and \eqref{17} we deduce that
\begin{align}\label{22}
Q_A=\delta_A-s\cdot\tr(A)&=\delta_A-s\cdot\delta_A(f)/f\in\Ann_{D_V[s]}f^s,\\
\nonumber Q_A^*=\delta_A^*+s\cdot\tr(A)&=\delta_A^*-s\cdot\delta_A^*(f^*)/f^*\in\Ann_{D_{V^*}[s]}(f^*)^s,
\end{align}
for $A\in\gg_D$ when identifying $A=d\rho_D(A)$.
So by \eqref{8} and \eqref{22}, we have 
\begin{equation}\label{23}
P=f^*(\p_x)\cdot f(x)-b_D(s)\in\Ann_{D_V[s]}f^s\ni Q_{A_k}=Q_k
\end{equation}
for $k=1,\dots,n$. 
Now let $\calD_V$ denote the sheaf of algebraic differential operators on $V$. 
Then, at any point $p\in U=V\backslash D$, we have  
\begin{equation}\label{84}
\Ann_{\calD_{V,p}[s]}f^s=\ideal{Q_1,\dots,Q_n}_{\calD_{V,p}[s]}.
\end{equation}
Indeed, as $f\in\calO_{V,p}^*$, conjugating by $f^s$ reduces this statement to $\Ann_{\calD_{V,p}[s]}1^s=\ideal{\delta_1,\dots,\delta_n}_{\calD_{V,p}[s]}$ which holds true since $\ideal{\delta_1,\dots,\delta_n}_{\calO_{V,p}}=\ideal{\p_{x_1},\dots,\p_{x_n}}_{\calO_{V,p}}$.
The equality \eqref{84} gives rise to an exact sequence 
\[
\xymat@C=48pt{0\ar[r]&{\mathcal R}\ar[r]&\calD_U[s]^n\ar[r]^-{(Q_1,\dots,Q_n)}&\Ann_{\calD_U[s]}f^s\ar[r]&0}.
\]
As $U$ is affine and the sheaf ${\mathcal R}$ is quasicoherent, we have $H^1(U,{\mathcal R})=0$ (see \cite[Prop.~2.5.4]{Meb89} and \cite[Thm.~3.5]{Har77}) and this implies with \eqref{23} that
\begin{equation}\label{25}
f^m\cdot P\in\ideal{Q_1,\dots,Q_n}_{D_V[s]},
\end{equation}
for some $m\in\NN$.
Let $\calF$ denote the Fourier transform.
Then, by \eqref{21}, \eqref{19}, and \eqref{22}, we have 
\begin{align}\label{72}
\calF(Q_A(s))
&=\calF(x^tA^t\p_x-s\cdot\tr(A))\\
\nonumber &=-\p_y^tA^ty-s\cdot\tr(A)\\
\nonumber &=-y^tA\p_y-(s+1)\cdot\tr(A)\\
\nonumber &=\delta_A^*+(-s-1)\cdot\tr(A)=Q_A^*(-s-1).
\end{align}
Combining \eqref{22}, \eqref{25}, and \eqref{72} yields that
\begin{align*}
\Ann_{D_{V^*}[s]}(f^*)^{-s-1}\ni\calF(f^m\cdot P)
&=\calF(f(x))^m\cdot\calF(f^*(\p_x)\cdot f(x)-b_D(s))\\
&=(-1)^{nm}f(\p_y)^m\cdot((-1)^n\cdot f^*(y)\cdot f(\p_y)-b_D(s))
\end{align*}
and hence, using \eqref{8} and \eqref{9},
\begin{align*}
0&=f(\p_y)^m\cdot((-1)^n\cdot f^*(y)\cdot f(\p_y)-b_D(s))\cdot(f^*)^{-s-1}\\
&=f(\p_y)^m\cdot((-1)^n\cdot b_{D^*}(-s-2)-b_D(s))\cdot(f^*)^{-s-1}\\
&=((-1)^n\cdot b_{D^*}(-s-2)-b_D(s))\cdot f(\p_y)^m\cdot(f^*)^{-s-1}\\
&=((-1)^n\cdot b_D(-s-2)-b_D(s))\cdot b_{D^*}(-s-2)\cdots b_{D^*}(-s-m-1)\cdot(f^*)^{-s-1}.
\end{align*}
Thus, $(-1)^n\cdot b_D(-s-2)-b_D(s)=0$ as claimed.
\end{proof}

We conclude the following from Lemma~\ref{24} using Theorem~\ref{56}, and from \eqref{74}, \eqref{34}, and \cite{Kas76}.

\begin{thm}\label{6}
For any reductive prehomogeneous determinant $D\subseteq V$ in dimension $\dim V=n$, the $b$-function $B_D(s)\equiv b_D(s)$ has degree $n$ and rational negative roots symmetric about $-1$.
In particular, $-1$ is the only integer root.
\end{thm}

\begin{rmk}
In \cite[Prob.~3.2]{CU04}, Castro-Jim\'enez and Ucha-Enr\'\i quez ask whether $-1$ is the only integer root of the $b$-function for any linear free divisor.
Theorem~\ref{6} gives a positive answer to this question in the case of reductive prehomogeneous determinants.
In Theorem~\ref{37} in Section~\ref{26}, we shall also settle the case of certain nonreductive linear free divisors.
\end{rmk}

\section{Examples}\label{70}

Up to dimension $4$ all linear free divisors have been classified in \cite[\S6.4]{GMNS06}.
Figure~\ref{5} lists their $b$-functions computed using {\tt Macaulay 2} \cite{M2}.
Here $\gg_2$ is the non-Abelian Lie algebra of dimension $2$ and $\gg_3$ is characterized as having $2$-dimensional Abelian derived algebra $\gg'_3$, on which the adjoint action of a basis vector outside $\gg'_3$ is semi-simple with eigenvalues $1$ and $2$ (see \cite[Ch.~I, \S4]{Jac79}). 
Observe that, except for having degree greater than $n$, also the nonreductive divisors in Table~\ref{5} satisfy the statement of Theorem~\ref{6}.

\begin{longtable}{|c|c|c|c|c|}
\caption{$b$-functions of linear free divisors in dimension $n\le4$}\label{5}\\
\hline
$n$ & $f$ & $\gg_D$ & reductive? & $\Spec b(s)$ \\
\hline
\hline
$1$ & $x$ & $\CC$ & Yes & $-1$ \\
\hline
\hline
$2$ & $xy$ & $\CC^2$ & Yes & $-1,-1$ \\
\hline
\hline
$3$ & $xyz$ & $\CC^3$ & Yes & $-1,-1,-1$ \\
\hline
$3$& $(y^2+xz)z$ & $\bb_2$ & No & $-\frac{5}{4},-1,-1,-\frac{3}{4}$ \\
\hline
\hline
$4$ & $xyzw$ & $\CC^4$ & Yes & $-1,-1,-1,-1$ \\
\hline
$4$ & $(y^2+xz)zw$ & $\CC\oplus\bb_2$ & No & $-\frac{5}{4},-1,-1,-1,-\frac{3}{4}$ \\
\hline
$4$ & $(yz+xw)zw$ & $\CC^2\oplus\gg_2$ & No & $-\frac{4}{3},-1,-1,-1,-\frac{2}{3}$ \\
\hline
$4$ & $x(y^3-3xyz+3x^2w)$ & $\CC\oplus\gg_3$ & No & \begin{minipage}{3cm}\begin{center}$-\frac{7}{5},-\frac{4}{3},-\frac{6}{5},$ $-1,-1,-1,$ $-\frac{4}{5},-\frac{2}{3},-\frac{3}{5}$\end{center}\end{minipage} \\
\hline
$4$ & \begin{minipage}{3cm}\begin{center}$y^2z^2-4xz^3-4y^3w+18xyzw-27w^2x^2$\end{center}\end{minipage}& $\gl_2(\CC)$ & Yes & $-\frac{7}{6},-1,-1,-\frac{5}{6}$ \\
\hline
\end{longtable}

Discriminants of quiver representations are at the origin of linear free divisors and form a rich source of reductive determinants.
We shall briefly review the basics of the theory to understand the following three examples and refer to \cite{BM06} for more details.

A \emph{quiver} $Q=(Q_0,Q_1)$ is a directed finite graph with vertex set $Q_0$ and edge set $Q_1\subseteq Q_0\times Q_0$. 
To a dimension vector $\d=(d_i))\in\NN^{Q_0}$ we associate the \emph{representation space} 
\[
\Rep(Q,\d)=\prod_{(i,j)\in Q_1}\Hom(\CC^{d_i},\CC^{d_j}).
\]
The \emph{quiver group}  
\[
\GL(Q,\d)=\prod_{i\in Q_0}\GL_{d_i}(\CC)
\]
acts on $\Rep(Q,\d)$ through the \emph{quiver representation}
\begin{gather*}
\rho_{Q,\d}\colon\GL(Q,\d)\to\GL(\Rep(Q,\d)),\\
\rho_{Q,\d}((g_i)_{i\in Q_0})(\varphi_{(i,j)})_{(i,j)\in Q_1}=(g_j\circ\varphi_{(i,j)}\circ g_i^{-1})_{(i,j)\in Q_1}.
\end{gather*}
Note that $Z=\CC\cdot(I_{d_i})_{i\in Q_0}$, where $I_i\in\GL_i(\CC)$ denotes the unit matrix, lies in the kernel of $\rho_{Q,\d}$.
The \emph{Tits form} $q\colon\CC^{Q_0}\to\CC$ is defined by 
\[
q((x_i)_{i\in Q_0})=\sum_{i\in Q_0}x_i^2-\sum_{(i,j)\in Q_1}x_ix_j.
\]
The condition $\dim G=\dim V$ in Definition~\ref{58} for the discriminant $D$ of $(G,\rho,V)=(\GL(Q,\d)/Z,\rho_{Q,\d},\Rep(Q,\d))$ to define prehomogeneous determinant is equivalent to $q(\d)=1$ and reductivity is automatic if $D\ne V$ \cite[\S4]{GMNS06}. 

The following example of a linear free divisor is the discriminant of the star quiver studied in \cite[Exa.~5.3]{GMNS06}.
We use the microlocal calculus developed in \cite{SKO80} to determine its $b$-function and then {\tt Macaulay 2} \cite{M2} to verify our result.

\begin{exa}\label{66}
Consider the \emph{star quiver} $Q$ shown in Figure~\ref{76} with one sink, three sources, and dimension vector $\d=(2,1,1,1)$ with $q(\d)=1$.
\begin{figure}[h]
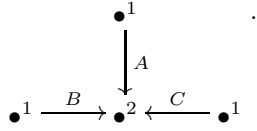

\caption{The star quiver in Example~\ref{66}}\label{76}
\[
\xymat{
&\bullet^1\ar[d]^A&\\
\bullet^1\ar[r]^B&\bullet^2&\bullet^1\ar[l]_C}.
\]
\end{figure}
The quiver group $G=\GL(Q,\d)=\GL_2(\CC)\times(\CC^*)^3\subseteq\GL_2(\CC)\times\GL_3(\CC)$ acts on the $6$-dimensional representation space $\Rep(Q,\d)=\CC^{2\times3}=\{(A\ B\ C)\}$ of $2\times3$-matrices through $\rho=\rho_{Q,\d}$ given by
\[
\rho(g)x=g_1\cdot x\cdot g_2^{-1},\quad g=(g_1,g_2).
\]
The infinitesimal vector fields for this group action are given by
\begin{equation}\label{52}
\xi_{i,j}=x_{i,1}\p_{j,1}+x_{i,2}\p_{j,2}+x_{i,3}\p_{j,3},\quad\xi_k=x_{1,k}\p_{1,k}+x_{2,k}\p_{2,k}
\end{equation}
where $x_{i,j}$ are coordinates on $\CC^{2\times3}$ and $\p_{i,j}=\p_{x_{i,j}}$.
The discriminant is a reductive linear free divisor $D$ defined by the product of maximal minors 
\[
f(x)=(x_{1,2}x_{2,3}-x_{2,2}x_{1,3})(x_{1,1}x_{2,3}-x_{2,1}x_{1,3})(x_{1,1}x_{2,2}-x_{2,1}x_{1,2})
\]
with associated character 
\begin{equation}\label{47}
\chi(g)=\det(g_1)^3\det(g_2)^{-2}.
\end{equation}

We first enumerate the different orbits by specifying a generic element $x_0$.
We denote by $O_{i,j}$ (or $O_i$ if there is only one) an orbit of dimension $i$ of a point $x_0^{i,j}$ (or $x_0^j$), and by $\Lambda_{i,j}$ (or $\Lambda_i$) its conormal space.
In order to calculate the dimension of the orbit of $x_0$ it is sufficient to determine the tangent space $T_{x_0}=d\rho(\gg)x_0$ which is spanned by the evaluation of the vector fields \eqref{52} at $x_0$.

\begin{asparaenum}

\item\label{50} Orbits of rank $2$:

\begin{asparaenum}[(\ref{50}.a)]
\item $1$ open orbit:
\[
x_0^6=\begin{pmatrix}
1 & 0 & 1 \\ 
0 & 1 & 1
\end{pmatrix}
\]
\item $3$ orbits of rank $2$ and dimension $5$ with all columns nonzero:
\[
x_0^{5,1}=
\begin{pmatrix}
1 & 0 & 0 \\ 
0 & 1 & 1
\end{pmatrix},\quad
x_0^{5,2}=
\begin{pmatrix}
0 & 1 & 0 \\ 
1 & 0 & 1
\end{pmatrix},\quad
x_0^{5,3}=
\begin{pmatrix}
0 & 0 & 1 \\ 
1 & 1 & 0
\end{pmatrix}.
\]

\item $3$ orbits of rank $2$ and dimension $4$ with exactly one zero column:
\[
x_0^{4,1}=
\begin{pmatrix}
0 & 1 & 0 \\ 
0 & 0 & 1
\end{pmatrix},\quad 
x_0^{4,2}=
\begin{pmatrix}
1 & 0 & 0 \\ 
0 & 0 & 1
\end{pmatrix},\quad 
x_0^{4,3}=
\begin{pmatrix}
1 & 0 & 0 \\ 
0 & 1 & 0
\end{pmatrix}. 
\]

\end{asparaenum}

\item\label{51} Orbits of rank $1$:

\begin{asparaenum}[(\ref{51}.a)]

\item $1$ orbits of rank $1$ and dimension $4$:
\[
x_0^{4,0}=
\begin{pmatrix}
1 & 1 & 1 \\ 
0 & 0 & 0
\end{pmatrix}
\]

\item $3$ orbits of rank $1$ and dimension $3$:
\[
x_0^{3,1}=
\begin{pmatrix}
0 & 1 & 1 \\ 
0 & 0 & 0
\end{pmatrix},\quad
x_0^{3,2}=
\begin{pmatrix}
1 & 0 & 1 \\ 
0 & 0 & 0
\end{pmatrix},\quad
x_0^{3,3}=
\begin{pmatrix}
1 & 1 & 0 \\ 
0 & 0 & 0
\end{pmatrix}.
\]

\item $3$ orbits of rank $1$ and dimension $2$:
\[
x_0^{2,1}=
\begin{pmatrix}
1 & 0 & 0 \\ 
0 & 0 & 0
\end{pmatrix},\quad
x_0^{2,2}=
\begin{pmatrix}
0 & 1 & 0 \\ 
0 & 0 & 0
\end{pmatrix},\quad
x_0^{2,3}=
\begin{pmatrix}
0 & 0 & 1 \\ 
0 & 0 & 0
\end{pmatrix}.
\]

\end{asparaenum}

\item $1$ Orbit of rank $0$: $x_0^0=0$.

\end{asparaenum}

For each of the above points $x_0=x_0^{i,j}$, we shall now compute the action $\rho_{x_0}$ of the isotropy group $G_{x_0}$ on the normal space $V_{x_0}=V/T_{x_0}$ to the tangent space $T_{x_0}$ to the orbit of $x_0$.
We omit the trivial cases of the open and of the zero orbit and describe $V_{x_0}$ by identifying it with a representative subspace in $V$.

\begin{asparaenum}

\item $x_0=x_0^{5,1}$:
\begin{gather*}
G_{x_0}=\left\{g=
\begin{pmatrix}
a & 0 \\ 
0 & 1
\end{pmatrix}
\times (a,1,1)
\right\},\quad
T_{x_0}=\ideal{\p_{1,1},\p_{2,1},\p_{2,2},\p_{2,3},\p_{1,2}+\p_{1,3}},\\
V_{x_0}=\left\{x=
\begin{pmatrix}
0 & 0 & u \\ 
0 & 0 & 0
\end{pmatrix}
\right\},\quad
\rho_{x_0}(g)x=a\cdot x.
\end{gather*}

\item $x_0=x_0^{4,3}$: 
\begin{gather*}
G_{x_0}=\left\{g=
\begin{pmatrix}
1 & 0 \\ 
0 & a
\end{pmatrix}
\times (1,a,b) 
\right\},\quad
T_{x_0}=\ideal{\p_{1,1},\p_{1,2},\p_{2,1},\p_{2,2}},\\
V_{x_0}=\left\{x=
\begin{pmatrix}
0 & 0 & u \\ 
0 & 0 & v
\end{pmatrix}
\right\},\quad
\rho_{x_0}(g)x=
\frac 1b
\begin{pmatrix}
0 & 0 & u\\
0 & 0 & av
\end{pmatrix}.
\end{gather*}

\item $x_0=x_0^{4,0}$:
\begin{gather*}
G_{x_0}=
\left\{g=
\begin{pmatrix}
1 & a \\ 
0 & b
\end{pmatrix}
\right\},\quad
T_{x_0}=\ideal{\p_{1,1},\p_{1,2},\p_{1,3},\p_{2,1}+\p_{2,2}+\p_{2,3}},\\
V_{x_0}=
\left\{x=
\begin{pmatrix}
0 & 0 & 0 \\ 
0 & u & v
\end{pmatrix}
\right\},\quad
\rho_{x_0}(g)x=b\cdot x.
\end{gather*}
This is the only nonprehomogeneous normal action.

\item $x_0=x_0^{3,3}$:
\begin{gather*}
G_{x_0}=
\left\{g=
\begin{pmatrix}
1 & a \\ 
0 & b
\end{pmatrix}
\times (1,1,c) 
\right\},\quad
T_{x_0}=\ideal{\p_{1,1},\p_{1,2},\p_{2,1}+\p_{2,2}},\\
V_{x_0}=
\left\{x=
\begin{pmatrix}
0 & 0 & u \\ 
0 & v & w
\end{pmatrix}
\right\},\quad
\rho_{x_0}(g)x=
\begin{pmatrix}
0&\frac{u+aw}{c}\\ 
bv&b\frac{w}{c}
\end{pmatrix}.
\end{gather*}
This is a prehomogeneous action with nonreduced discriminant $\{vw^2=0\}$.

\item $x_0=O_{2,1}$:
\begin{gather*}
G_{x_0}=
\left\{g=
\begin{pmatrix}
1 & a \\
0 & b
\end{pmatrix}
\times (1,c,d) 
\right\},\quad
T_{x_0}=\ideal{\p_{11},\p_{21}},\\
V_{x_0}=
\left\{x=
\begin{pmatrix}
0 & r & s \\ 
0 & t & u
\end{pmatrix}
\right\},\quad
\rho_{x_0}(g)x=
\begin{pmatrix}
0 & \frac{r+at}{c} & \frac{s+au}{d} \\
0 & \frac{bt}{c} & \frac{bu}{d} 
\end{pmatrix}
\end{gather*}
This is a prehomogeneous action with reduced discriminant $\{wr(ur-vw)=0\}$

\end{asparaenum}

The holonomy diagram in Figure~\ref{49} shows the incidence relations between the good orbit conormals subject to intersections in codimension one. 
Recall that $\Lambda_{4,0}$ is not good as it is not prehomogeneous.
\begin{figure}[h]
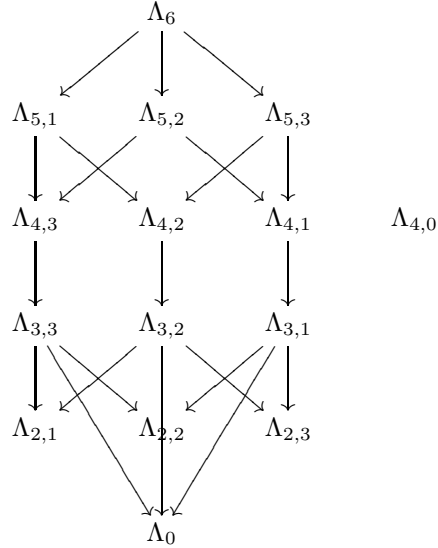

\caption{Holonomy diagram of Example~\ref{66}}\label{49}
\[
\xymat{
&\Lambda_{6}\ar[dl] \ar[d] \ar[dr]& & \\
\Lambda_{5,1}\ar[d]\ar[dr]&\Lambda_{5,2}\ar[dl]\ar[dr]&\Lambda_{5,3}\ar[dl]\ar[d] &\\
\Lambda_{4,3}\ar[d]&\Lambda_{4,2}\ar[d]&\Lambda_{4,1}\ar[d] & \Lambda_{4,0} \\ 
\Lambda_{3,3}\ar[d]\ar[dr]\ar[ddr]&\Lambda_{3,2}\ar[dl]\ar[dr]\ar[dd] & \Lambda_{3,1}\ar[dl]\ar[d] \ar[ddl] & \\
\Lambda_{2,1}&\Lambda_{2,2}& \Lambda_{2,3} & \\
&\Lambda_0 & &  
}
\]
\end{figure}
We shall apply the method in \cite{SKO80} to the chain
\begin{equation}\label{54}
\xymat{
\Lambda_6\ar[r]& \Lambda_{5,1}\ar[r]&\Lambda_{4,3}\ar[r]&\Lambda_{3,3}\ar[r]&\Lambda_0.
}
\end{equation}
An elementary local calculation shows that all the intersections in this chain are smooth and transversal at the generic point which is the situation of \cite[Cor.~7.6]{SKO80}. 
Therefore it remains to calculate $\ord_{\Lambda}f^s$ for all $\Lambda$ in the chain \eqref{54} using the following formula from \cite[Prop.~4.14]{SKO80}:
\begin{equation}\label{48}
\ord_{\Lambda}{f^s}=-m_{\Lambda}s-{\frac{\mu_{\Lambda}}{2}}=d\chi(A_0)-\tr_{V_{x_0}^*}d\rho^*_{x_0}(A_0)+\frac12\dim V_{x_0}^*.
\end{equation}
Here $A_0=(A_0^1,A_0^2)\in\gg_{x_0}$ and $(x_0,y_0)\in\Lambda$ is generic such that $d\rho_{x_0}(A_0)y_0=y_0$.
By \eqref{47}, we have
\[
d\chi(A_0)=3\cdot\tr(A_0^1)-2\cdot\tr(A_0^2),\quad\tr_{V_{x_0}^*}d\rho^*_{x_0}(A_0)=-\tr_{V_{x_0}}d\rho_{x_0}(A_0)
\]
which can be determined from the preceding computations.
Note that the contragredient action $\rho^*$ on the dual space $(\CC^{2\times3})^*$ identified with $\CC^{3\times2}$ by the trace pairing is given by
\[
d\rho^*(g)\xi={g_2}\xi{g_1}^{-1},\quad g=(g_1,g_2).
\]
One computes
\begin{gather*}
\ord_{\Lambda_6}f^s=0,\quad
\ord_{\Lambda_{5,1}}f^s=-s-\frac12,\quad
\ord_{\Lambda_{4,3}}f^s=-2s-1,\\
\ord_{\Lambda_{3,3}}f^s=-5s-\frac52,\quad
\ord_{\Lambda_0}f^s=-6s-3 
\end{gather*}
Using the formula in \cite[Cor.~7.6]{SKO80} this yields
\begin{gather*}
\frac{b_{\Lambda_{5,1}}}{b_{\Lambda_6}}=s+1,\quad
\frac{b_{\Lambda_{4,3}}}{b_{\Lambda_{5,1}}}=-s-\frac12+2s+1+\frac12=s+1,\\
\frac{b_{\Lambda_{3,3}}}{b_{\Lambda_{4,3}}}=(3s+2)(3s+3)(3s+4),\quad
\frac{b_{\Lambda _0}}{b_{\Lambda _{3,3}}}=-5s-\frac52+6s+3+\frac12=s+1.
\end{gather*}
From this we conclude that
\[
B_D(s)=b_{\Lambda_0}(s)=\left(s+\frac23\right)(s+1)^4\left(s+\frac43\right)
\]
is the $b$-function of $D$ which is also confirmed by {\tt Macaulay 2} \cite{M2}.
\end{exa}

The next two examples are reductive nonreduced prehomogeneous determinants arising from quiver representations. 

\begin{exa}\label{67}
We modify the star quiver in Example~\ref{66} by adding an additional source.
Consider the resulting star quiver $\wt D_3$ in Figure~\ref{77} and the dimension vector $\d=(2,2,1,1,1)$ with $q(\d)=1$.
\begin{figure}[h]
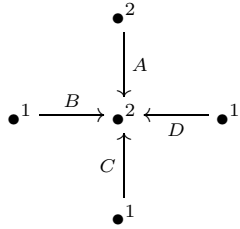

\caption{The star quiver in Example~\ref{67}}\label{77}
\[
\xymat{
&\bullet^2\ar[d]^A\\
\bullet^1\ar[r]^B&\bullet^2&\bullet^1\ar[l]^D\\
&\bullet^1\ar[u]^C}
\]
\end{figure}
The quiver group $\GL(Q,\d)=\GL_2(\CC)\times\GL_2(\CC)\times(\CC^*)^3$ acts on a $10$-dimensional representation space $\Rep(Q,\d)=\CC^{2\times2}\times\CC^{3\times2}=\{(A,(B\ C\ D))\}$ and the orbit of
$\left(
\begin{pmatrix}1&0\\0&1\end{pmatrix},
\begin{pmatrix}1&0&1\\0&1&1\end{pmatrix}
\right)$
is open.
The nonreduced discriminant $D$ defined by the function
\[
f(x)=(x_{1,1}x_{2,2}-x_{1,2}x_{2,1})^2\cdot(x_{1,3}x_{2,4}-x_{1,4}x_{2,3})\cdot(x_{1,3}x_{2,5}-x_{1,5}x_{2,3})\cdot(x_{1,4}x_{2,5}-x_{1,5}x_{2,4})
\]
is a reductive prehomogeneous determinant.
It is a union $D=D_1\cup D_2$ where $D_1$ and $D_2$ are defined by the functions
\begin{align}\label{75}
f_1(x)&=(x_{1,1}x_{2,2}-x_{1,2}x_{2,1})^2,\\
\nonumber f_2(x)&=(x_{1,3}x_{2,4}-x_{1,4}x_{2,3})\cdot(x_{1,3}x_{2,5}-x_{1,5}x_{2,3})\cdot(x_{1,4}x_{2,5}-x_{1,5}x_{2,4}).
\end{align}
Since $D_1$ is the square of a quadratic form and $D_2$ is the reductive linear free divisor from Example~\ref{66}, we have
\begin{align*}
B_{D_1}(s)&=(2s+1)(s+1)^2(2s+3)\\
B_{D_2}(s)&=\left(s+\frac23\right)(s+1)^4\left(s+\frac43\right).
\end{align*}
As the functions in \eqref{75} depend on separate sets of variables, $B_D(s)$ divides $B_{D_1}(s)\cdot B_{D_2}(s)$ and hence 
\[
B_D=B_{D_1}\cdot B_{D_2}=\left(s+\frac12\right)\left(s+\frac23\right)(s+1)^6\left(s+\frac43\right)\left(s+\frac32\right)
\]
as $\deg B_D=10=\deg B_{D_1}\cdot\deg B_{D_2}$ by Theorem~\ref{6}.
For the reduced discriminant $D_\red$, we loose the symmetry property.
Indeed, one computes
\[
B_{D_\red}(s)=\left(s+\frac23\right)(s+1)^5\left(s+\frac43\right)\left(s+2\right)
\]
using the same degree argument.
\end{exa}

\begin{exa}\label{68}
Consider the quiver representation associated to the quiver $\wt A_n$ in Figure~\ref{78} and the dimension vector $\d=(2,1,\dots,1)$ with $q(\d)=1$.
\begin{figure}
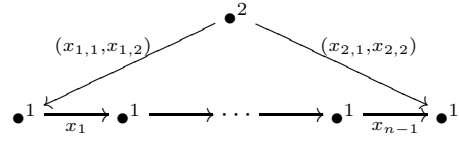

\caption{The quiver in Example~\ref{68}}\label{78}
\[
\xymat{
&&\bullet^2\ar[dll]_-{(x_{1,1},x_{1,2})}\ar[drr]^-{(x_{2,1},x_{2,2})}&&\\
\bullet^1\ar[r]_{x_1}&\bullet^1\ar[r]&\cdots\ar[r]&\bullet^1\ar[r]_{x_{n-1}}&\bullet^1}
\]
\end{figure}
The representation space is $(n+3)$-dimensional and the orbit of 
\[
\left(\begin{pmatrix}x_{1,1}&x_{1,2}\\x_{2,1}&x_{2,2}\end{pmatrix},(x_1,\dots,x_{n-1})\right)=
\left(\begin{pmatrix}1&0\\0&1\end{pmatrix},(1,\dots,1)\right)
\]
is open.
The nonreduced discriminant $D$ defined by the function
\[
f(x)=\det\begin{pmatrix}x_{1,1}&x_{1,2}\\x_{2,1}&x_{2,2}\end{pmatrix}^2\cdot x_1\cdots x_{n-1}
=(x_{1,1}x_{2,2}-x_{1,2}x_{2,1})^2\cdot x_1\cdots x_{n-1}
\]
is a reductive prehomogeneous determinant.
As in Example \ref{67}, one computes
\begin{align*}
B_D(s)&=\left(s+\frac12\right)(s+1)^{n+1}\left(s+\frac32\right),\\
B_{D_\red}(s)&=(s+1)^n\left(s+2\right), 
\end{align*}
and the reduced discriminant does not have the symmetry property.
\end{exa}

\section{$b$-functions of special linear free divisors}\label{26}

The symmetry of the $b$-functions of nonreductive prehomogeneous determinants in Section~\ref{70} motivates to try to weaken the hypotheses of Theorem~\ref{6}.
In the case of linear free divisors, this turns out to be possible.

Consider a linear free divisor $D$ defined by $f$ as in Definition~\ref{58}.
In \cite[Thm.~6.1]{GMNS06}, we proved that $\aa_D$ has a basis $A_2,\dots,A_n=S_2,\dots,S_s,N_1,\dots,N_{n-r}$ as in the proof of Theorem~\ref{56}:
the $d\rho_D(S_i)$ are semisimple and commute, the $d\rho_D(N_i)$ are nilpotent, and the $N_i$ are homogeneous with respect to the $S_i$ with rational weight. 
Then, as in \eqref{79}, we have 
\begin{equation}\label{29}
d\chi_f(A_i)=\tr\circ d\rho_D(A_i)+\sum_{j=1}^n[A_i,A_j]/A_j,\quad i=1,\dots,n.
\end{equation}

Assume now that $D^*\ne V^*$.
Then the analog of \eqref{29} for $D^*$ reads
\begin{equation}\label{30}
d\chi_{f^*}(A_i)=\tr\circ d\rho_D^*(A_i)+\sum_{j=1}^n[A_i,A_j]/A_j,\quad i=1,\dots,n.
\end{equation}
As $\tr\circ d\rho_D^*(A_i)=-\tr\circ d\rho_D(A_i)$, the difference of \eqref{29} and \eqref{30} is
\begin{equation}\label{39}
d\chi_f(A_i)-d\chi_{f^*}(A_i)=2\tr\circ d\rho_D(A_i),\quad i=1,\dots,n.
\end{equation}
Using \eqref{1}, we conclude from \eqref{39} the following formulas.

\begin{lem}\label{31}
For a linear free divisor $D$ with $D^*\ne V^*$,
\[\pushQED{\qed}
\chi_f\cdot\chi_{f^*}^{-1}={\det}^2\circ\rho_D,\quad d\chi_f-d\chi_{f^*}=2\tr\circ d\rho_D.\qedhere
\]
\end{lem}

As an immediate consequence of \eqref{1} and Lemma~\ref{31}, we obtain an equivalence of \eqref{27} and \eqref{17}.

\begin{prp}\label{35}
For a linear free divisor $D$ with $D^*\ne V^*$, we have $\chi_D=\det\circ\rho_D$ if and only if $\chi_{f^*}=\chi_f^{-1}$.\qed
\end{prp}

A relation of \eqref{8}, \eqref{9}, and \eqref{17}, can be established for general prehomogeneous determinants.

\begin{prp}\label{69}
Let $D$ be a prehomogeneous determinant with $D^*\ne V^*$.
Then \eqref{8} is equivalent to \eqref{17} and implies \eqref{9}.
\end{prp}

\begin{proof}
The classical proof essentially works in our setting:

The existence of $b_f(s)$ and $b_{f^*}$ satisfying \eqref{8} follows from \eqref{17} by the character argument in \cite[Prop.~2.22]{Kim03} which also gives the converse implication.

In the reductive case,
\[
f(x)=\sum_{|\alpha|=n}f_\alpha\cdot x^\alpha,\quad f^*(y)=\sum_{|\alpha|=n}f_\alpha\cdot x^\alpha
\]
have complex conjugate coefficients $f_\alpha=\ol{f_\alpha^*}$ in unitary coordinates. 
This is not necessary to show that
\[
b_f(0)=f^*(\p_x)\cdot f(x)=\sum_{|\alpha|=n}f_\alpha\cdot f^*_\alpha\cdot\alpha!=f(\p_y)\cdot f^*(y)=b_{f^*}(0)
\]
and similarly 
\[
b_f(0)\cdots b_f(m-1)=f^*(\p_x)^m\cdot f(x)^m=f(\p_y)^m\cdot f^*(y)^m=b_{f^*}(0)\cdots b_{f^*}(m-1)
\]
for any $m\in\NN$.
The equality in \eqref{9} follows as in \cite[Prop.~2.23]{Kim03}.
\end{proof}

In \cite[Def.~2.1]{GMS08}, a linear free divisor is called \emph{special} if \eqref{27} holds true.
Using this terminology, we conclude from Lemma~\ref{24} and Propositions~\ref{35} and \ref{39} the following result analogous to Theorem~\ref{6}.

\begin{thm}\label{36}
Let $D\subseteq V$ be a special linear free divisor in dimension $\dim V=n$ with $D^*\ne V^*$.
Then the $b$-function $b_D(s)$ has degree at most $n$ and roots symmetric about $-1$.\qed
\end{thm}

Under the stronger assumption of regularity we can prove more. 

\begin{thm}\label{37}
Let $D\subseteq V$ be a regular special linear free divisor in dimension $\dim V=n$.
Then the $b$-function $B_D(s)\equiv b_D(s)$ has degree at most $n$ and negative rational roots symmetric about $-1$.
In particular, $-1$ is the only integer root.
\end{thm}

\begin{proof}
Note that the hypothesis $D^*\ne V^*$ in Theorem~\ref{36} becomes redundant for regular $D$ by \cite[Thm.~2.16]{Kim03}.

By \cite[Prop.~4.6]{SKO80}, regularity of $D$ is equivalent to $V^*=T_0^*V$ being a good holonomic variety.
Under this latter hypothesis the arguments in \cite[Lem.~2.5.7]{Gyo91} still yield \eqref{34}.
\end{proof}

Note that in Theorem~\ref{37} regularity of $f$ is equivalent to $\deg b_D(s)=n$.
Indeed, the degree $n$ coefficient of $b_f(s)$ is given by $f^*(\grad\log f)\cdot f$ by \cite[Prop.~2.13.(2)]{Kim03} and $\grad\log f\colon V\to V^*$ is a $G$-equivariant map by \cite[Prop.~2.13.(1)]{Kim03}.

\section{Euler homogeneity and normal representations}

The condition of (strong) Euler homogeneity arose in the context of the logarithmic comparison theorem in \cite[Conj.~1.4]{CMNC02} (see also \cite{CMN96} and \cite{GS06}) and has been studied in \cite[\S7]{GMNS06} in the case of linear free divisors.
It is conceivable that all linear free divisors are (strongly) Euler homogeneous, but there is neither a proof nor a counter-example.

\begin{dfn}
Let $D\subseteq V$ be a reduced hypersurface.
A vector field $\eps\in\Der(-\log D)$ is called an \emph{Euler vector field} for $D$ if $\eps(f)=f$ for some defining equation $f$ of $D$.
It is called an \emph{Euler vector field at $p\in D$} if in addition $\eps(p)=0$.
The hypersurface $D$ is called \emph{(strongly) Euler homogeneous at $p\in D$} if there is an Euler vector field at $p$ defined locally at $p$.
By (strong) Euler homogeneity of $D$ we mean that this property holds for all $p\in D$.
\end{dfn}

The following statement is implicit in the proof of \cite[Lem.~7.5]{GMNS06}.

\begin{lem}\label{41}
If $D$ is (locally) defined by $f$ and Euler homogeneous with Euler vector field $\eps$ then $D$ is strongly Euler homogeneous if $\eps(p)\in\Der(-\log f)(p)$ for all $p\in D$.
\end{lem}

Assume now that $D$ is a linear free divisor.
Then the preceding observation can be formulated as follows.

\begin{lem}\label{10}
Strong Euler homogeneity of of a linear free divisor $D$ means that the $G_D$-orbits and $A_D$-orbits in $D$ coincide.
\end{lem}

We shall now give a sufficient condition for Euler homogeneity of $D$ at $x_0\in D$.
Let $G_{D,x_0}$ be the stabilizer of $x_0$ in $G_D$ with Lie algebra $\gg_{D,x_0}$ and consider the representation $\rho_{D,x_0}\colon G_{D,x_0}\to V_{x_0}$ induced by $\rho_D$ on the normal space $V_{x_0}=V/d\rho_D(\gg_D)x_0$ to the tangent space of the $G_D$-orbit of $x_0$ at $x_0$.

The localization $f_{x_0}\in\Sym(V^*)$ of $f$ at $x_0$ \cite[Def.~6.7]{SKO80} is the homogeneous nonzero polynomial defined by 
\[
f(x_0+\veps x')\equiv\veps^k\cdot f_{x_0}(x')\mod\ideal{\veps^{k+1}}.
\]
By \cite[Def.~6.8]{SKO80}, $f_{x_0}(x')$ induces a relative invariant $f_{x_0}(y)$ of the representation $(G_{D,x_0},\rho_{D,x_0},V_{x_0})$ with character $\chi_{f_{x_0}}=\chi_D$ where $y=y_1,\dots,y_k$ are coordinates on $V_{x_0}$.
Note however that the latter might not be a prehomogeneous vector space.

In the context of linear free divisors, Definition~\ref{58} applied to $(G,\rho,V)=(G_{D,x_0},\rho_{D,x_0},V_{x_0})$, as in Definition~\ref{60}, gives another natural candidate $D_{x_0}=D$ for such an invariant.

\begin{dfn}
Assume that $A_1,\dots,A_k$ form a basis of $\gg_{D,x_0}$ and define $D_{x_0}\subseteq V_{x_0}$ by 
\[
f'_{x_0}(y)=\det(d\rho_{D,x_0}(A_1)y,\dots,d\rho_{D,x_0}(A_k)y)\in\Sym(V_{x_0}^*).
\]
\end{dfn}

Then $f'_{x_0}\ne0$ or $D_{x_0}\ne V_{x_0}$ is equivalent to $(G_{D,x_0},\rho_{D,x_0},V_{x_0})$ being a prehomogeneous vector space.

\begin{lem}\label{46}
Let $D$ be a linear free divisor and assume that $x_0\in D$ such that $D_{x_0}\ne V_{x_0}$. 
Then $f'_{x_0}(y)=f_{x_0}(y)$.
\end{lem}

\begin{proof}
Note that $A_{k+1}x_0,\dots,A_nx_0$ form a basis of the tangent space $d\rho_D(\gg_D)x_0$ to the $G_D$-orbit of $x_0$ at $x_0$.
We may assume that $d\rho_D(\gg_D)x_0=0\times\CC^{n-k}$ and then that
\begin{equation}\label{43}
(A_{k+1}x_0,\dots,A_nx_0)=
\begin{pmatrix}
0&\dots&0\\
\vdots&&\vdots\\
0&\dots&0\\
1&\dots&0\\
\vdots&\ddots&\vdots\\
0&\dots&1\\
\end{pmatrix}
\end{equation}
We may then choose $y=x'_1,\dots,x'_k$.
Using \eqref{3} and \eqref{43}, we compute
\begin{align}\label{44}
f(x_0+\veps x')&=\det(A_1(x_0+\veps x'),\dots,A_n(x_0+\veps x'))\\
\nonumber&\equiv\veps^k\cdot\det(A_1x',\dots,A_kx',A_{k+1}x_0,\dots,A_nx_0)\mod\ideal{\veps^{k+1}}.
\end{align}
Denoting by $A'_i$ the matrix consisting of the upper $k$ rows of $A_i$, \eqref{43} implies that
\begin{align}\label{45}
\det(A_1x',\dots,A_kx',A_{k+1}x_0,\dots,A_nx_0)&=\det(A'_1x',\dots,A'_kx')\\
\nonumber&\equiv\det(d\rho_{D,x_0}(A_1)y,\dots,d\rho_{D,x_0}(A_k)y)\\
\nonumber&=f'_{x_0}(y)\mod\ideal{x_{k+1},\dots,x_n}
\end{align}
By hypothesis, $f'_{x_0}\ne0$ and the claim follows from \eqref{44} and \eqref{45}.
\end{proof}

\begin{prp}
Let $D$ be a linear free divisor and assume that $x_0\in D$ such that $D_{x_0}\ne V_{x_0}$. 
Then $D$ is Euler homogeneous at $x_0$.
\end{prp}

\begin{proof}
By assumption $(G_{D,x_0},\rho_{D,x_0},V_{x_0})$ is a prehomogeneous vector space and $f_{x_0}(y)$ is a nonzero relative invariant with character $\chi_D$.
Thus, there must be an infinitesimal vector field as in \eqref{21} that is not tangent to the level sets of $f_{x_0}(y)$.
By \eqref{20}, this means that $\chi_D(B)=1$ for some $B\in\gg_{x_0}$.
Using \eqref{2} and \eqref{55} we conclude that $B=A_1-A$ with $A\in\aa_D$ and hence $\rho_D(A_1)x_0=\rho_D(A)x_0$ by definition of $\gg_{x_0}$.
Then the claim follows from Lemma~\ref{41}.
\end{proof}

\section{Koszul freeness and Euler homogeneity}

Koszul-freeness is a natural finiteness condition on logarithmic vector fields and is also closely related to the logarithmic comparison theorem (see \cite{Tor04} and \cite{CN05}).
In geometric terms, \emph{Koszul freeness} of a free divisor $D\subseteq V$ can be defined as the local finiteness of the \emph{logarithmic stratification} by maximal integral submanifolds along logarithmic vector fields as defined in \cite[\S3]{Sai80}.
As shown in \cite[Thm.~7.4]{GMNS06}, this condition can be formulated algebraically as the symbols of a basis of $\Der(-\log D)\subseteq\calD_V$ forming a regular sequence in $\gr_F\calD_V=\calO_V\otimes\Sym(V)$ where $F$ denotes the order filtration on $\calD_V$.
Avoiding the technical details, this equivalence can be understood as follows:
At each point $p\in V$, $\Der(-\log D)(p)$ spans the tangent space of the logarithmic stratum through $p$.
Thus, the \emph{logarithmic variety} in $T^*V$ defined by the symbols of local bases of $\Der(-\log D)$ is the union of conormals to the logarithmic strata (see \cite[(3.16)]{Sai80}).
It is of (co-)dimension $n$ exactly where the symbols of $n$ local basis elements of $\Der(-\log D)$ form a regular sequence in the Cohen--Macaulay coordinate ring $\calO_V\otimes\Sym(V)$ of $T^*V$ (see \cite[1.8]{CN02}).

It is tempting to introduce a stronger version of Koszul freeness by imposing local finiteness of the logarithmic stratification of $\Der(-\log f)$ on $D$. 
But unless $D$ is strongly Euler homogeneous, which reduces this condition to ordinary Koszul freeness by Lemma~\ref{41}, it is not clear that the condition in question depends only on $D$. 

For a linear free divisor $D$, the logarithmic stratification is the stratification by orbits of $G_D$ consisting of smooth locally closed algebraic varieties (see \cite[Prop.~8.3]{Hum75}). 
Using the action of $\CC^*\subseteq G_D$, one can see that local finiteness of this stratification is equivalent to finiteness.
Moreover, there is a natural choice of defining equation which we use to introduce the desired strong version of Koszul freeness for this class of divisors.

\begin{dfn}\label{12}
We call a linear free divisor $D$ \emph{Koszul free} if the stratification of $D$ by orbits of $G_D$ is finite, and \emph{strongly Koszul free}, if the same holds for the stratification by orbits of $A_D$.
\end{dfn}

Note that the $G_D$-orbit $V\backslash D$ is a union of infinitely many $A_D$-orbits by reasons of dimension.
The algebraic definition of Koszul freeness above translates to strong Koszul freeness as follows.

\begin{prp}\label{71}
A linear free divisor $D$ is Koszul free if and only if the symbols of a basis of $\gg_D$ form a regular sequence in $\gr_FD_V$.
It is strongly Koszul free if and only if the symbols of a basis of $\aa_D$ and $f$ form a regular sequence in $\gr_FD_V$.
\end{prp}

\begin{proof}
Let $\delta_2,\dots,\delta_n$ be a basis of $\aa_D$ and denote by $\sigma_F\colon D_V\to\gr_FD_V$ the symbol map for the order filtration $F$ on $D_V$.
As $\gr_FD_V=\Sym(V^*\oplus V)$ is a polynomial ring hence Cohen--Macaulay, we have to show that the variety $\Var(f,\sigma_F(\delta_2),\dots,\sigma_F(\delta_n))$ has (co-)dimension $n$ exactly if the stratification of $D$ by orbits of $A_D$ is finite.
As $\aa_D$ spans the tangent spaces of the orbits, $\Var(\sigma_F(\delta_2),\dots,\sigma_F(\delta_n))$ is the union of conormals to the orbits and $\Var(f)=T^*V|_D$.
The second claim follows and the proof of the first one is similar.
\end{proof}

Let us now relate the properties of strong Euler homogeneity, Koszul, and strong Koszul freeness for linear free divisors.

\begin{prp}\label{11}
The following implications hold for any linear free divisor $D$.
\begin{enumerate}[(a)]
\item\label{11a} Strong Koszul freeness implies Koszul freeness.
\item\label{11b} Strong Koszul freeness implies strong Euler homogeneity.
\item\label{11c} For Koszul free $D$, strong Euler homogeneity and strong Koszul freeness are equivalent.
\end{enumerate}
\end{prp}

\begin{proof}
\eqref{11a} is clear by Definition~\ref{12} and \eqref{11c} follows from \eqref{11b} and Lemma~\ref{10}.

For $0\le k\le n-1$, let $T_k=\{p\in D\mid\rk\aa_D(p)\le k\}$ (see \cite[\S7]{GMNS06} and \cite[Def.~3.12]{Sai80}).
Then $T_k$ is a homogeneous algebraic variety defined by $f$ and the $(k+1)\times(k+1)$-minors of the matrix $(\delta_i)_{i=2,\dots,n}$ and equals the union of $A_D$-orbits of dimension at most $k$.
As the Euler vector field $\eps$ is the infinitesimal generator of the $\CC^*$-action that defines homogeneity, $\eps$ is tangent to $T_k$ and to all $G_D$-orbits.
Therefore, $T_k$ is also a union of $G_D$-orbits.
So if $D$ is strongly Koszul free, $T_k\backslash T_{k-1}$ is the finite union of all $k$-dimensional $A_D$-orbits each of which is also a $G_D$-orbit.
This proves \eqref{11b}.
\end{proof}

\section{Logarithmic comparison theorem}

Let $D\subseteq V$ be a reduced hypersurface, let $U=V\smallsetminus D$, and let $j:U\to V$ be inclusion.
Then the de~Rham morphism
\[
\Omega_V^\bullet(*D)\to\R j_*\CC_U
\]
is a quasi-isomorphism.
This so-called \emph{Grothendieck's comparison theorem} is the core of the proof of Grothendieck's algebraic de~Rham theorem \cite{Gro66} in which $D$ plays the role of a divisor at infinity for a compactification.

By analogy with this comparison theorem, one says that the \emph{logarithmic comparison theorem (LCT)} holds for $D$ if the inclusion
\[
\Omega_V^\bullet(\log D)\hookrightarrow\Omega_V^\bullet(*D)
\]
is a quasi-isomorphism.
For isolated singularities the LCT can be explicitly characterized in the quasihomogeneous case (see \cite{HM98}) and there are partial results for the nonquasihomogeneous case (see \cite{Sch07b}).
An algebraic version of the LCT is conjectured to hold for all hyperplane arrangements in characteristic $0$ (see \cite[Conj.~3.1]{Ter78} and \cite[\S1]{WY97}). 
A global version of the LCT has been proved in \cite[Thm.~1.6]{GMNS06} for reductive linear free divisors and it is conceivable that LCT holds for any linear free divisor.

The LCT has been studied intensively in the case of free divisors.
The main conjecture in this case is the following from \cite[Conj.~1.4]{CMNC02} and has been proved in \cite[Thm.~1.6]{GS06} for $n\le3$.

\begin{cnj}[Calder{\'o}n-Moreno et al.]
For a free divisor, LCT implies strong Euler homogeneity.
\end{cnj}

Even for the converse implication, which holds for plane curves by \cite[Thm.~3.5]{CMNC02}, a counter-example is missing.
A weaker converse has been proved in \cite{CMN96}: Any locally quasihomogeneous free divisor satisfies the LCT.
By definition, such divisors admit at any point a defining equation which is quasihomogeneous with respect to strictly positive weights on a local coordinate system.
Under the additional \emph{Spencer type} hypothesis the strictness can be dropped except for one weight (see \cite[rem.~3.11]{CGHU07}).

The probably deepest (but equally unexplicit) characterization of the LCT in terms of $\calD_V(\log D)$-modules is given in \cite[Cor.~4.3]{CN05}.

\begin{thm}[Calder{\'o}n-Moreno and Narv{\'a}ez-Macarro]\label{13}
A free divisor $D$ satisfies the logarithmic comparison theorem if 
\begin{enumerate}[(a)]
\item\label{13a} $\calD_V\overset L\otimes_{\calD_V(\log D)}\calO_V(D)$ is concentrated in degree $0$ and
\item\label{13b} $\calD_V\otimes_{\calD_V(\log D)}\calO_V(D)\to\calO_V(*D)$ is injective.
\end{enumerate}
\end{thm}

Condition \ref{13}.\eqref{13a} is fulfilled for Koszul free $D$ and condition \ref{13}.\eqref{13b} is equivalent to $\Ann_{\calD_V}(f^{-1})$ being generated by order one operators.
This latter condition is studied in \cite{Tor04} where the following is proved in \cite[Cor.~1.8]{Tor04} for the (local analytic) Koszul free case.

\begin{thm}[Torrelli]\label{14}
Let $D$ be a Koszul free germ defined by $f$.
Then $\Ann_{\calD_V}(f^{-1})$ is generated by order one operators if and only if 
\begin{enumerate}[(a)]
\item\label{14a} $D$ is Euler homogeneous,
\item\label{14b} $-1$ is the only integer root of the (local) $b$-function of $D$, and
\item\label{14c} $\Ann_{\calD_V}(f^s)$ is generated by order one operators.
\end{enumerate}
Moreover, if \eqref{14a} holds and $f$ admits an Euler vector field then \eqref{14c} is equivalent to:
\begin{enumerate}[(a)]\setcounter{enumi}{3}
\item\label{14d} $f$ and the symbols of a basis of $\Der(-\log f)$ form a $\gr_F\calD_V$-regular sequence.
\end{enumerate}
\end{thm}

Because of the $\CC^*$-action, the (global) $b$-function of a linear free divisor $D$ coincides with the local $b$-function of its germ $(D,0)$ at the origin (see \cite[Lems.~2.5.3-2.5.4]{Gyo91}). 
By Proposition~\ref{71} and flatness of $\calO_V$ over $\Sym(V^*)$, Koszul freeness of a linear free divisor $D$ is equivalent to Koszul freeness of the germ $(D,0)$.
By the same argument, strong Koszul freeness of $D$ is equivalent to condition \eqref{14d} in Theorem~\ref{14} for $(D,0)$.
Therefore, combining our Theorem~\ref{6} and Proposition~\ref{11}.\eqref{11c} with Theorems~\ref{13} and \ref{14}, we conclude the following result.

\begin{thm}
A Koszul free reductive linear free divisor satisfies the logarithmic comparison theorem if and only if it is Euler homogeneous.
\end{thm}

\bibliographystyle{amsalpha}
\bibliography{rlfd}

\end{document}